\documentclass[a4paper]{amsart}

\DeclareMathAlphabet{\mathpzc}{OT1}{pzc}{m}{it}

\usepackage{amssymb}
\usepackage[hyphens]{url} 
\usepackage{color}
\usepackage[english]{babel}
\newcommand{\m}{\mathcal{M}}
\newcommand{\mo}{\mathcal{M}_0}
\newcommand{\mt}{\mathcal{M}_t}
\newcommand{\tb}{\bar{t}}
\newcommand{\xb}{\bar{x}}

\newcommand{\wb}{\overline{W}}

\newcommand{\pb}{\bar{p}}
\newcommand{\qb}{\bar{q}}
\newcommand{\qti}{\tilde{Q}}
\newcommand{\kti}{\tilde{K}}
\newcommand{\hti}{\tilde{h}}
\newcommand{\Hti}{\tilde{H}}
\newcommand{\bti}{\tilde{b}}
\newcommand{\Bti}{\tilde{B}}
\newcommand{\Ati}{|\tilde{A}|}
\newcommand{\tr}{\text{Tr}}

\newcommand{\HK}{\mathbb{H}_{a}^{n+1}}

\newcommand{\hh}{\mathpzc{h}}
\newcommand{\co}{\text{\normalfont co}_a}
\newcommand{\ta}{\text{\normalfont ta}_a}
\newcommand{\ck}{\text{\normalfont c}_a}
\newcommand{\sk}{\text{\normalfont s}_a}

\newtheorem{theorem}{Theorem}[section]
\newtheorem{lem}[theorem]{Lemma}
\newtheorem{prop}[theorem]{Proposition}
\newtheorem{cor}[theorem]{Corollary}

\theoremstyle{definition}

\theoremstyle{remark}

\numberwithin{equation}{section}

\title[Volume preserving non homogeneous flow in hyperbolic space]{Volume preserving non homogeneous mean curvature flow in hyperbolic space}

\author{\sc Maria Chiara Bertini \and Giuseppe Pipoli} 

\date{}

\begin{document}

\maketitle

\begin{abstract}

We study a volume/area preserving curvature flow of hypersurfaces that are convex by horospheres  in the hyperbolic space, with velocity given by a generic positive, increasing  function of the mean curvature, not necessarly homogeneous. For this class of speeds we prove the exponential convergence to a geodesic sphere. The proof is ispired by \cite{CaMi}  and is based on the preserving of the convexity by horospheres that allows to  bound the inner and outer radii and to give uniform bounds on the curvature by maximum principle arguments. In order to deduce the exponential trend, we study the behaviour of a suitable ratio associated to the hypersurface that converges exponentially in time to the value associated to a geodesic sphere.

\end{abstract}

\vspace{1cm}
 \bigskip
\section{Introduction}

Let $\HK$ be the hyperbolic  space of constant sectional curvature $-a^2<0$ and let us take a smooth oriented, compact and without boundary hypersurface $F_0:\m\rightarrow\HK$. We consider a family of maps $F:\m\times[0,T)\rightarrow\HK$, evolving according the law: \smallskip
\begin{equation}\label{fl}
\left\{
\begin{array}{l}
\partial_t F(x,t)=[-\phi(H(x,t))+h(t)]\nu(x,t) \medskip\\
F(x,0)=F_0(x),\\
\end{array}
\right.
\end{equation}
where:
\begin{itemize}
\item $H$ and $\nu$ denote respectively the mean curvature and the outer unit normal  vector of the evolving hypersurface $\mt:=F(\m,t)$.
 \item $\phi:[0,+\infty)\rightarrow \mathbb{R}$ is a continuous function, $C^2$ differentiable in $(0,+\infty)$  such that\medskip

  \begin{trivlist}
         \item $i)$ $\phi(\alpha)>0, \hspace{0.5cm} \phi'(\alpha)>0\hspace{5mm}\forall \alpha>0$; \medskip
         \item $ii)$ $\displaystyle \lim_{\alpha\to \infty}\phi(\alpha)=\infty$;
         \item $iii)$  $ \displaystyle\lim_{\alpha \to \infty}\frac{\phi'(\alpha)\alpha^2}{\phi(\alpha)}=\infty$; \medskip
         \item $iv)$ $\phi''(\alpha)\alpha\geq-2\phi'(\alpha)\hspace{5mm}\forall \alpha>0$. \medskip
         \end{trivlist}

 \item The function $h(t)$ is either defined as
 \begin{equation}\label{vpr}h(t):=\frac{1}{A_t}\int_{\mt} \phi(H)d\mu
 \end{equation}
 or as
 \begin{equation}\label{apr}h(t):=\frac{\int_{\mt} H\phi(H)d\mu}{\int_{\mt} Hd\mu}.
 \end{equation}
\end{itemize}
where $A_t=\int_{\mt}d\mu_t$ denotes the area of $\mt$.

  The choice of $h$  is made in order to keep the volume enclosed by $\mt$ constant in case \eqref{vpr}, and in order to keep the area of $\mt$ constant in case \eqref{apr}. Flows of this form are sometimes called {\em constrained} curvature flows, while the corresponding ones without the $h(t)$ term will be referred to as {\em standard} flows.

In this paper we restrict our attention to the class of $\hh$-convex hypersurfaces, that turns to be a good choice when the ambient manifold is the hyperbolic space.  Roughly speaking, we will see that $\hh$-convexity is strong enough to offset  the negative curvature of the ambient manifold and to  be preserved along the flow.
The main result proved in this paper is the following.

\begin{theorem}\label{mt}Let  $F_0:\m\rightarrow \HK$, with $n \geq 1$, be a smooth embedding of an oriented, compact   $n$-dimensional  manifold without boundary, such that $F_0(\m)$ is $\hh$-convex. Then the flow \eqref{fl} with $h(t)$ given by \eqref{vpr} (resp. \eqref{apr}) has a unique smooth solution, which exists for any time $t\in[0,\infty)$. The solution is $\hh$-convex for any time and converges smoothly and exponentially, as $t\to\infty$, to a geodesic sphere that encloses the same volume (resp. has the same area) as the initial datum $\mo$.
\end{theorem}

A similar  flow was recently studied by the first author and Sinestrari in \cite{BeSi} for strictly convex hypersurfaces of the Euclidean space. Since convexity is weaker than  $\hh$-convexity, the authors in \cite{BeSi} needed  a certain behaviour of the velocity and its derivative at zero.  We do not require these hypotheses on $\phi$ and $\phi'$. Then we recover from \cite{BeSi} a very large class of velocities, as linear combinations of powers with degree greater than zero, logarithms and exponentials. On the other hand, we get some extra examples given by functions with a behaviour at zero not admitted in \cite{BeSi}.

 Constrained curvature flows have been intensively studied in recent years. For an overview on curvature flows in the Euclidean ambient manifold, see for example \cite{CaSi}. Our main source of inspiration is the paper of Cabezas-Rivas and Miquel \cite{CaMi}, that we generalize. Theorem 1.2 in  \cite{CaMi} in fact is a particular case of our Theorem \ref{mt} when $\phi(H)=H$ and $h$ is taken as in \eqref{vpr}.   
After \cite{CaMi}, the evolution of $\hh$-convex hypersurfaces was explored in many contests. For example, in \cite{Ma,WaXi} some mixed volume preserving  flows were considered, where the velocity is assumed to be a degree one homogeneous function of the principal curvatures. In those cases too the 
authors have the convergence to a geodesic sphere. Also there are some results on curvature flows in the hyperbolic space with velocity of degree greater then one,  but the conditions required on the initial datum are  stronger than $\hh$-convexity, as in 
 \cite{GLW}. In our paper instead we obtain in particular the convergence to a geodesic sphere for the flow with velocity $\phi=H^k, k>1, $ requiring just a condition that is the natural equivalent of convexity in Eucledean setting. Also, we include velocities given by non homogeneous functions that satisfy some very general properties.  Non-homogeneous flows have been sometimes studied in the Euclidean ambient space, in the standard or constraint case. We recall in particular the works of \cite{AlSi,BeSi,CT1,Sm}. We point out that, for the best of our knowledge,  a non homogeneous flow in a not Eucledean ambient space is considered here for the first time.

The paper is organized as follows. In Section 2 we collect some preliminaries and fix some notations. In Section 3 we use the maximum principle for tensors to prove that $\hh$-convexity is preserved along the flow. Then, by some well known  results valid for $\hh$-convex domains, we show that the inradius is uniformly bounded from both sides.  We use this property in Section  $4$  to bound the function $\phi$ and hence $H$. By the $\hh$-convexity then also the curvatures are bounded. In this way we can prove the long time existence of the solution and, by the uniform parabolicity of the flow, we also get the existence of a limit hypersurface. In Section $5$  we complete the proof of Theorem \ref{mt}. In the first part we get the smooth convergence by showing that the mean curvature tends uniformly to a constant value, thus the limit hypersurface has to be a geodesic sphere by a classical result by Alexandrov \cite{Al}. In the second part we prove that the convergence has an exponential rate using a method taken from \cite{Sch2} and \cite{GLW}. 

\section{Preliminaries}\label{Preliminaries}

\subsection*{Notations}
For every constant $a>0$, we denote with $\HK$ the hyperbolic space of dimension $n+1$ and constant sectional curvature $-a^2$, and let $\langle\cdot,\cdot\rangle$ be its standard Riemannian metric. We denote by $d_\mathbb{H}$ the hyperbolic distance between points induced by $\langle\cdot,\cdot\rangle$. Moreover, we put a bar over any geometrical quantity whenever it is referred to the ambient space $\HK$. Let $F:\m\rightarrow \HK$ be an embedded hypersurface with local coordinates $(x^1,\cdots, x^n)$. We endow $\m$ with the induced metric $g=(g_{ij})$ given by 
$$g_{ij}=\left\langle\frac{\partial F}{\partial x^i},\frac{\partial F}{\partial x^j}\right\rangle$$
We also denote respectively by $\nabla$ and $A=(h_{ij})$ the Levi-Civita connection and the second fundamental form of $\m$, while the measure is $d\mu=\sqrt{\det g_{ij}}\, dx$. The principal curvatures of $\m$ are the eigenvalue of $A$ with respect to $g$ and they are denoted by $\lambda_1,\dots,\lambda_n$. The mean curvature is $H=\lambda_1+\dots+\lambda_n$. 
. We denote by  $\Delta=g^{ij}\nabla_i\nabla_j$ the Laplace-Beltrami operator, where $g^{-1}=(g^{ij})$ is the inverse of the metric. As usual, we always sum on repeated indices, and we lower or lift tensor indices via $g$, e.g. the Weingarten operator is given by $h^i_j=h_{kj}g^{ik}$. Moreover the metric $g$ induces, in the usual way, a norm on tensors. For example the norm of the second fundamental form is $|A|^2=h_i^jh_j^i=\sum_i\lambda_i^2$.

\subsection*{Some results in hyperbolic geometry}
In this paper we restrict our attention to the class of $\hh$-convex hypersurfaces. We say that a hypersurface is \textit{convex by horospheres} (\textit{$\hh$-convex} for short) if it bounds a domain $\Omega$ such that at every point $p\in\m=\partial\Omega$ there exists a horophere of $\HK$ passing through $p$ such that $\Omega$ is contained in the region bounded by the horosphere. In \cite{BoMi} was proved that $\m$ is $\hh$-convex if and only if at any point $\lambda_i\geq a$ for all $i$. Note that this condition is stronger than convexity. 

We will use the following notations for the hyperbolic functions: for any $a>0$  

$$
\begin{array}{ll}
\sk(t)=\frac{\sinh(a t)}{a}, & \ck(t)=\cosh(a t),\\\\
\ta(t)=\frac{\sk(t)}{\ck(t)} & \co(t)=\frac{\ck(t)}{\text{s}_a(t)}
\end{array}
$$
Given a point $q\in\HK$, we set
$$
\begin{array}{l}
r_q(p)=d_{\mathbb{H}}(p,q) \hspace{0.5cm}\forall p\in\HK,\\
\\
\partial_{r_p}=\bar{\nabla}r_p.
\end{array}
$$


We recall that the inradius of a bounded domain $\Omega\subset\HK$ is  the biggest radius of a geodesic ball contained in $\Omega$. Such a geodesic ball is called inball. The following theorem is due to \cite{BoGaRe,BoMi,BoMi2,BoVl}.

\begin{theorem}\label{iperb2}
Let $\Omega$ be a compact $\hh$-convex domain of $\HK$, and let $q\in\Omega$ the center of a inball of $\Omega$. If $\rho$ is the inradius of $\Omega$, then
\begin{enumerate}
\item the maximal distance $\max d_{\HK}(q,\partial\Omega)$ between $q$ and the point in $\Omega$ satisfies the inequality
$$\max d_{\HK}(q,\partial\Omega)\leq\rho+a\ln\frac{(1+\sqrt{\ta\frac{\rho}{2}})^2}{1+\ta\frac{\rho}{2}}<
\rho+a\ln2$$
\item For any interior point $p$ of $\Omega$ and any boundary point $q\in\partial\Omega$, 
$$\langle\nu(q),\partial_{r_p}\rangle\geq a\ta(d_{\mathbb{H}}(p,\partial\Omega)),$$
where $\nu(q)$ is the outer normal vector to $\partial\Omega$ at $q.$ 
\end{enumerate}
\end{theorem}

\subsection*{Short time existence and evolution equations}
It is well known that a flow of the form \eqref{fl} 
 is parabolic if at any point
\begin{equation}\label{par}
\frac{\partial\phi}{\partial\lambda_i}>0,\hspace{1cm}i=1,\dots,n
\end{equation}
i.e. $\phi'>0$ at any point. This is guaranteed, by condition $i)$ on $\phi$,  because we assume that the initial datum is $\hh$-convex and then in particular $H>0$ at least 
 in  a small time interval. 
Parabolicity ensures the local existence and uniqueness of the solution.
Hence we have the following result, see
 \cite{Hu2, HuPo,Mc2} for more details.
\begin{theorem} Let $F_0:\m\rightarrow \HK$ be a smooth embedding of an oriented, compact  $n$-dimensional  manifold without boundary, such that $F_0(\m)$ is strictly mean convex. Then the flow \eqref{fl} has a unique smooth solution $\mt$ defined on a maximal time interval $[0,T)$.
\end{theorem}

In the next proposition we list the evolution equations for the main geometrical quantities associated with the flow \eqref{fl}, which can be computed similarly to \cite{Hu1}, see also \cite{AlSi,Sm}. For brevity, the dependence of the functions by space and time is omitted. We remind in particular that $\phi=\phi(H(x,t))$ is space and time dependent, while $h=h(t)$ is only time dependent.

\begin{prop}\label{eq-ev}

We have the following evolution equations for the flow \eqref{fl}:
\begin{align*}
&\partial_t g_{ij} = 2(-\phi +h)h_{ij},\\
&\partial_t g^{ij} = -2(-\phi +h)h^{ij},\\
&\partial_t \nu=\nabla\phi,\\
&\partial_t d\mu = H(-\phi+h)d\mu, \\
&\partial_t h_{ij} = \phi'\Delta h_{ij}-(\phi'H+\phi-h) h_{il}h_j^l+\phi'|A|^2h_{ij}+\phi''\nabla_iH\nabla_jH\\
&\phantom{\partial_t h_{ij} = }-a^2(\phi'H+\phi-h)g_{ij}+na^2\phi'h_{ij},\\
&\partial_t H = \phi'\Delta H +\phi''|\nabla H|^2+(\phi-h)|A|^2-na^2(\phi-h),\\
&\partial_t\phi = \phi'\Delta\phi+\phi'(\phi-h)|A|^2-na^2\phi'(\phi-h),\\
\end{align*}
\end{prop}

Let us denote by $\Omega_t$ the region enclosed by $\mt$. $A_t$ is the $n$-dimensional measure of $\mt$, while $V_t$ is the $(n+1)$-dimensional measure of $\Omega_t$.

\begin{lem}\label{l1}
Along the flow \eqref{fl} we have
$$
\begin{array}{lll}
1)& \frac{d}{dt}A_t\leq 0,& \frac{d}{dt}V_t\geq 0.\\
2)& a_1\leq A_t\leq a_2,& v_1\leq V_t\leq v_2,
\end{array}
$$
for some positive constants $a_1$, $a_2$, $v_1$ and $v_2$.
\end{lem}
\begin{proof}
1) The same proof of Lemma 3.1 of \cite{BeSi} holds.\\
2) It follows from part 1) and the isoperimetric inequality in $\HK$. 
\end{proof}

We finish this section giving some remarks on the function $\phi$. If $\phi$ is a convex function, properties $ii),iii)$ and $iv)$ on $\phi$ just follows from the convexity: $ii)$ and $iv)$ are trivial, and for $iii)$  we can write $\frac{\phi(\alpha)}{\alpha}=\frac{\phi(\alpha)-\phi(0)}{\alpha}+\frac{\phi(0)}{\alpha}$.  By the convexity of $\phi$, the first addendum of the right side is an increasing function,  then $\left(\frac{\phi(\alpha)}{\alpha}\right)'\geq-\frac{\phi(0)}{\alpha^2}$, which implies $\phi'(\alpha)\alpha\geq\phi(\alpha)-\phi(0)$. Finally,
$$\lim_{\alpha\to\infty}\frac{\phi'(\alpha)\alpha^2}{\phi(\alpha)}\geq
\lim_{\alpha\to\infty}\frac{\phi(\alpha)-\phi(0)}{\phi(\alpha)}\alpha=\infty$$
Using also property $ii)$ on $\phi$.

\section{Preserving of $\hh$-convexity and its consequences}
The goal of this section is to  bound uniformly in time some geometrical quantities that allow to control the shape of the hypersurface. In order to do this, the crucial step is to prove that  $\hh$-convexity is preserved by the flow.

\begin{prop}\label{pr-h}
Let $\mo$ be an $\hh$-convex hypersurface of $\HK$, then $\mt$ is $\hh$-convex for any time the flow \eqref{fl} is defined.
\end{prop}
\begin{proof}

Since $\mo$ is $\hh$-convex, we can consider a time interval $[0,T^*)$, with $T^*<T$,  such that $\mt$ is strictly convex for any time $t\in[0,T^*)$. Then we can define $b_i^j$, be the inverse matrix of $h_i^j$. Let us define the tensor $S_{ij}=b_{ij}-\frac{1}{a}g_{ij}$. We have that $\hh$-convexity is equivalent to the fact that $S_{ij}\leq 0$. The first step is to compute the evolution equation of $S_{ij}$. By Proposition \ref{eq-ev} we get:
\begin{eqnarray*}
\partial_t h_r^s & = & \phi'\Delta h_r^s+\phi''\nabla_r H\nabla^s H-\left(\phi'H-\phi+h\right)h_r^lh_l^s\\
&&-a^2\left(\phi'H+\phi-h\right)\delta_r^s+na^2\phi'h_r^s.
\end{eqnarray*}
Since $b_i^kh_k^j=\delta_i^j$ we can compute:
\begin{eqnarray*}
\Delta b_i^j & = & -b_s^jb_i^r\Delta h_r^s-2b_s^j\nabla_lb_i^r\nabla^lh_r^s.
\end{eqnarray*}
Therefore
\begin{eqnarray*}
\partial_t b_i^j & = & -b_s^jb_i^r\partial_t h_r^s\\
& = & \phi'\Delta b_i^j+2\phi'b_s^j\nabla_l b_i^r\nabla^l h_r^s-b_s^jb_i^r\phi''\nabla_rH\nabla^sH\\
 && -\left(\phi'H-\phi+h\right)\delta_i^j-\phi'|A|^2b_i^j-a^2\left(\phi'H+\phi-h\right)b_i^rb_r^j-na^2\phi' b_i^j
\end{eqnarray*}

Finally, by Proposition \ref{eq-ev}, we have:
\begin{eqnarray}\label{ev-S}
\nonumber\partial_t S_{ij} & = & \phi'\Delta S_{ij}+2\phi'b_{sj}\nabla_lb_i^r\nabla^lh_r^s-\phi''b_i^rb_{sj}\nabla_rH\nabla^sH\\
&&+\left(\phi'H-\phi+h\right)g_{ij}-\phi'|A|^2b_{ij}+a^2\left(\phi'H+\phi-h\right)b_i^rb_{rj}\\
\nonumber&&-na^2\phi'b_{ij}-2\left(\phi-h\right)g_{ij}+\frac{2}{a}\left(\phi-h\right)h_{ij}.
\end{eqnarray}

Analogously to the proof of Lemma $2.5$ in \cite{Sh}, we can use Codazzi equation and the fact that $\mt$ is strictly convex in order to estimate the gradient terms in \eqref{ev-S}:
$$
\begin{array}{l}
2\phi'b_{sj}\nabla_lb_i^r\nabla^lh_r^s-\phi''b_i^rb_{sj}\nabla_rH\nabla^sH\\
\phantom{aa} = -2\phi'b_j^sb_k^rb_i^c\nabla_lb_c^k\nabla^lh_{rs}-\phi''b_i^rb_{sj}\nabla_rH\nabla^sH\\
\phantom{aa} \leq-\dfrac{1}{H}\left(2\phi'+\phi''H\right)\left(b_i^r\nabla_rH\right)\left(b_{sj}\nabla^sH\right).
\end{array}
$$
Let $V$ be a null eigenvector of $S$ with unit norm. We  apply the reaction terms in the equation \eqref{ev-S} to $V$. What we get can be estimate as follows on $[0,T^*)$:
$$
\begin{array}{l}
-\dfrac{1}{a^2H}\left(2\phi'+\phi'' H\right)\left(\nabla_iHV^i\right)^2+\phi'\left(2H-\dfrac{|A|^2}{a^2}-na\right)\\
\phantom{aa} \leq -\dfrac{\phi'}{na}\left(H-na\right)^2\leq 0.
\end{array}
$$
In the last line, we used the hypothesis $iv)$ on the function $\phi$ and the fact that $|A|^2\geq\frac{H^2}{n}$ for any $n$-dimensional submanifold. Then $S_{ij}\leq0$ by the maximum principle for symmetric tensors that can be found in Theorem 9.1 of \cite{Ha}. Thus we have that, until $\mt$ is strictly convex, the solution is $\hh$-convex. If, by contradiction, it exists a first time $\bar{T}$ where the solution is not strictly convex, we can apply the previous argument on the interval $[0,\bar{T})$  to conclude that the solution is $\hh$-convex on $[0,\bar{T}]$. In particular, in $t=\bar{T}$  $h_{ij}\geq ag_{ij}$ holds. Then we find a contradiction.


%

\end{proof}

An immediate consequence of Proposition \ref{pr-h} is that, along the flow, $H\geq na>0$ at any space-time point.
By Proposition \ref{pr-h}, we are able to deduce some geometrical properties.  We begin proving that the inradius is uniformly bounded along the flow.

\begin{lem}\label{inradius}
Let $\rho_t$ be the inradius of the evolving domain $\Omega_t$ at time $t$. Then there are two positive constants $c_1$ and $c_2$ such that
$$
c_1\leq\rho_t\leq c_2.
$$
\end{lem}
\begin{proof}
The proof is the same of Lemma 4.1 of \cite{CaMi} with minor modifications in  case of the area preserving flow. Let $\psi$ be the inverse function of $s\mapsto vol(\mathbb {S}^n)\int_0^s\sk(l)dl$ and let $\xi$ be the inverse function of $s\mapsto s+a \ln \frac{\left(1+\ta\left(s/2 \right)\right)^2 }{1+\ta\left(s/2 \right) }$. Note that they are positive increasing functions.  
Proceeding like in Lemma 4.1 of \cite{CaMi} we get

$$
\xi(\psi(V_t))\leq\rho_t\leq\psi(V_t).
$$
By Lemma \ref{l1}, 
 the thesis follows with $c_1=\xi\left( \psi(v_1)\right) $ and $c_2=\psi(v_2)$. 
\end{proof}
As immediate corollary, using also the triangular inequality and Theorem \ref{iperb2} we obtain
\begin{cor}\label{bounddist}
For any $t\in[0,T)$, and $p,q\in\Omega_t$, we have
$$d_{\mathbb{H}}(p,q)<2(c_2+a\ln 2).$$
\end{cor}

\section{Long time existence}
 In this section our goal is to find a uniform upper bound on the curvatures of the hypersurfaces. This result will allow us to establish the long time existence of the solution together with the existence of a limit hypersurface. 
From Lemma \ref{inradius} we have a positive lower bound on the inner radii  $\rho_t$, so we can take $0<\rho\leq c_1$, with $c_1$ given in Lemma \ref{inradius}. 
\begin{lem} 
\label{pi}
There exists $\tau=\tau(a,n,\mo)>0$ with the following property: for all $(\qb,\tb)\in\Omega_{\tb}\times[0,T)$ such that $B(\qb,\rho)\subset\Omega_{\tb}$, then 
$$B(\qb,\rho/2)\subset\Omega_t\hspace{1cm}\forall t\in[\tb,\min\{\tb+\tau,T\})$$
\end{lem}
\begin{proof}
Let $\tb,\qb$ be like in the hypotheses. We consider the geodesic sphere centered at $\qb$ that evolves by the standard flow with initial datum $\rho$, i.e the radius $r_B(t)$ satisfies  
\begin{equation*}
   \begin{cases}
   r'_B(t)=-\phi(n\co(r_B))
   \\r_B(\tb)=\rho
   \end{cases}
\end{equation*}
We define $\tau$ as the time taken by the geodesic sphere of initial radius $\rho$ to contract its radius to $\rho/2$, i.e.
$$\tau=\int_{\rho/2}^\rho\frac{ds}{\phi(n\co(s))}$$
Notice that $\tau$ does not depend neither on $\qb$ nor from $\tb$, but only on $\rho$.
Let $r(x,t)=r_{\qb}(F_t(x))$. Then
$$\partial_t r=(h-\phi(H))\langle\nu,\partial_r\rangle.$$

Then we define the function $u(x,t):=r(x,t)-r_B(t)$ for $t\in[\tb,\tb+\tau]$, and  compute the evolution
\begin{equation}
\label{evf}
\partial_t u(x,t)=(h-\phi(H))\langle\nu,\partial_r\rangle+\phi(n\co(r_B)),
\end{equation}

where $\phi(H)=\phi(H(x,t))$. 
 Suppose that there exists a first time $t^*$ such that for some $x^*\in\m$ the ball  touches the hypersurface $\mathcal{M}_{t^*}$ at the point $F(x^*,t^*)$. Then we have 
$$n\co(r_B(t^*))\geq H(x^*,t^*)\hspace{1cm}\partial_t u(x^*,t^*)\leq0 
$$
 Thus, by the monotonicity of $\phi$ and the fact that $\langle\nu,\partial_r\rangle\leq1$, we obtain from \eqref{evf}
$$\partial_t u(x^*,t^*)\geq h\langle\nu,\partial_r\rangle>0$$

From this contradiction we get the result.
\end{proof}
It is also useful to define, as in \cite{CaMi}, the function:
$$\sigma_t(r_q):=\sk(r_q)\langle\nu,\partial_{r_q}\rangle$$
with $q$ a given point in $\Omega_t$. Following analogous calculations as in \cite{CaMi}, we get the evolution of $\sigma$:
$$\partial_t\sigma=\phi'\Delta\sigma+\phi'\sigma|A|^2+(h-\phi-\phi'H)\ck(r_q)$$
\begin{lem} \label{c} Given any $\tb\in[0,T)$ we have that, in $[\tb,\min\{\tb+\tau, T\})$:
\begin{enumerate}
\item there exist constants $C,D>0$ such that $$C\leq r_{\qb}\leq D;$$
\item taken $c:=\frac{a\sk(C)\ta(C)}{2}$ , $$\sigma_t(\qb)-c\geq c.$$
\end{enumerate}
where $\qb$ is taken as as in Lemma \ref{pi}
\end{lem}
\begin{proof}\
\begin{enumerate}\vspace{1mm}
\item  As a consequence of Lemma \ref{inradius} and Lemma \ref{pi}, on the time interval $[\tb,\min\{\tb+\tau, T\})$ we have $\frac{c_1}{2}\leq\frac{\rho}{2}\leq r_{\qb}$. On the other side, by Corollary \ref{bounddist}, $r_{\qb}\leq 2(c_2+a\ln2)$.
\item It follows by Theorem $2.2$ part $2)$ and some trivial computations.
\end{enumerate}
\end{proof}
\begin{prop}\label{ubv} There exists a positive constant $c_3=c_3(\mo,n,a)$ such that
$$\phi(H)\leq c_3\hspace{1cm}\text{on }[0,T).$$
\end{prop}

\begin{proof}
On any time interval $[\tb,\min\{\tb+\tau, T\})$, we consider
$$W(x,t):=\frac{\phi(H(x,t))}{\sigma_t(\qb)-c}$$
with $\qb$ as in Lemma \ref{pi} and $c$ as in Lemma \ref{c}.
Standard computations show that
\begin{eqnarray*}
(\partial_t-\phi'\Delta )W  &=& \frac{2\phi'}{\sigma-c}\langle\nabla W,\nabla\sigma\rangle-h\frac{\phi'}{\sigma-c}(|A|^2-na^2)-W\frac{{h}}{\sigma-c}\ck(r_{\pb})\\
&& +\left(1+\frac{\phi' H}{\phi}\right)W^2\ck(r_{\pb})-\frac{c}{\sigma-c}\phi'|A|^2W-na^2\phi'W
\end{eqnarray*}
By virtue of the  $\hh$-convexity we have $|A|^2-na^2\geq0$, then 
$$(\partial_t-\phi'\Delta )W \leq\frac{2\phi'}{\sigma-c}\langle\nabla W,\nabla\sigma\rangle+\left(1+\frac{\phi' H}{\phi}\right)W^2\ck(D)-\frac{c\phi'H^2}{n(\sigma-c)}W$$
where we also used the fact that $r_{\pb}\leq D$ and $|A|^2\geq\frac{H^2}{n}$.
We define
$$\wb(t):=\sup_{\mt}W(x,t)\hspace{1.5cm}X(t):=\{x\in\m | W(x,t)=\wb(t)\}$$
then 
$$D_+ W\leq\ck(D)\wb^2+\wb\sup_{X(t)}\frac{\phi'H}{\sigma-c}\left\{\ck(D)-\frac{cH}{n}\right\}$$
 Let $\tilde{C}>0$ big enough such that
\begin{equation}\label{Ctilde}
\begin{cases}
\tilde{C}\geq\frac{3n}{c}\ck(D)\\
\frac{c}{\phi(\tilde{C})}\leq\tau
.
\end{cases}
\end{equation}

Suppose that there exists a time $t^*$ such that  $\wb(t^*) \geq \phi(\tilde{C})/c$. Then, using the bound $\sigma-c \geq c$ and the monotonicity of $\phi$ we have that $H(x^*,t^*) \geq \tilde{C}$ for any $x^* \in X(t^*)$.  Notice that the first condition of \eqref{Ctilde} implies that if $H\geq\tilde{C}$, then $\ck(D)-\frac{cH}{n}\leq-\frac{2cH}{3n}$.
Then, we get at time $t=t^*$
\begin{align*}
D_+\wb&\leq\ck(D)\wb^2-\frac{2c}{3n}\wb\sup_{X(t^*)}\frac{\phi'H^2}{\sigma-c}\\
&=\wb^2\sup_{X(t^*)}\left\{\ck(D)-\frac{2c\phi'H^2}{3n\phi}\right\}\\
\end{align*}
By condition $iii)$ on $\phi$ , we can suppose $\tilde{C}$ sufficiently large such that $H\geq\tilde{C}$ implies
$$\ck(D)-\frac{2c\phi'H^2}{3n\phi}\leq-1$$
Thus at $t=t^*$ we have
$$D_+\wb\leq-\wb^2$$

Standard comparison argument then implies that 
\begin{equation}\label{w1}
W\leq\max\left\{\max_{\mo}W,\frac{\phi(\tilde{C})}{c}\right\}\hspace{1cm}\text{on  }[0,\min\{\tau,T\})
\end{equation}
in the case $\tb=0$, and
\begin{equation*}
W\leq\max\left\{\frac{1}{t-\tb},\frac{\phi(\tilde{C})}{c}\right\}\hspace{1cm}\text{on  }[\tb,\min\{\tb+\tau,T\})
\end{equation*}
for a general $\tb$. Then we also have
\begin{equation}\label{w2}
W\leq\frac{\phi(\tilde{C})}{c}\hspace{0.5cm}\text{on  }\left[\tb+\frac{c}{\phi(\tilde{C})},\min\{\tb+\tau,T\} \right).
\end{equation}
Since $\tb$ is arbitrary, combining \eqref{w1} and \eqref{w2} and
using the second condition of \eqref{Ctilde} that allows to cover the entire time interval, we obtain
$$W\leq\max\left\{\max_{\mo}W,\frac{\phi(\tilde{C})}{c} \right\}\hspace{1cm}\text{on  }t \in[0,T),$$
which implies the assertion, since $\phi \leq (\sk(D)-c) W$. 
\end{proof}
\begin{cor} The quantities $H,h$ and $|A|$ are uniformly bounded along the flow.
\end{cor}
\begin{proof}
By  property $ii)$ on $\phi$ , an upper bound on $\phi$ implies a bound on $H$. Then, by  Property \ref{ubv}, $H$ is uniformly bounded. The boundedness of $h$ also follows from the boundedness of $\phi$. Thanks to $\hh$-convexity, $|A|\leq H$, and so $|A|$ is bounded too. 
\end{proof}



From the bounds on $H$ it follows that also $\phi'$ is uniformly bounded from both sides, and then the flow is uniformly parabolic. Then, following the proof of Theorem $3.10$ in \cite{BeSi}, we can bound uniformly the curvatures and all its derivatives. Thus, we obtain the following theorem.
\begin{theorem}
The solution $\mt$ of the flow \ref{fl} exists for any time. Moreover, up to time subsequences and  space isometries, $\mt$ converges to a smooth limit $\m_{\infty}$.
\end{theorem}

\section{Exponential convergence to a sphere}

We will prove that the limit hypersurface $\m_{\infty}$ has to be a geodesic sphere by showing that the mean curvature tends to a constant value.  Then, we will show that the rate of the convergence is exponential. From this, we will deduce in particular that the hypersurfaces converge to a geodesic sphere with no need to add isometries.
\subsection*{Smooth convergence}
\begin{prop}\label{hc} The velocity $\phi(H)$ tends uniformly to $h$, i.e.
$$\lim_{t\to\infty}\max_{\mt}|\phi(H(x,t))-h(t)|=0$$
\end{prop}

\begin{proof}
The proof is the same of Proposition 4.4 of \cite{BeSi} for the analogous problem in the Euclidean space. We report the proof for completeness.

We begin with the volume preserving case.
For any $t$, let $\bar{H}(t)$ such that $\phi(\bar{H}(t))=h(t)$. Then we compute
\begin{align*}
\frac{d}{dt}A_t&=\int_{\mt}Hh \, d\mu-\int_{\mt}H\phi(H) \, d\mu\\
&=\int_{\mt}(H-\bar{H})(\phi(\bar{H})-\phi(H)) \, d\mu\\
&=-\int_{\mt}|H-\bar{H}||\phi(H)-\phi(\bar{H})| \, d\mu.
\end{align*}
Now, using the bound on $\phi'$ we obtain
\begin{align*}
\frac{d}{dt}A_t&\leq - \frac{1}{\sup\phi'} \int_{\mt}|\phi(H)-\phi(\bar{H})|^2 \, d\mu.\\
&=- \frac{1}{\sup\phi'}\int_{\mt}|\phi(H)-h|^2 \, d\mu.
\end{align*}
Suppose that $|\phi(H)-h|=\alpha$ for some $\alpha>0$ at some point $(\xb,\tb)$. The  bounds on the derivatives of the curvatures imply that $H$ is uniformly Lipschitz continuous, and then there exists a radius $r(\alpha)$, not depending on $(\xb,\tb)$, such that 
$$|\phi(H)-h|>\frac{\alpha}{2}\hspace{1cm}\text{on }B((\xb,\tb),r(\alpha))$$
where $B((\xb,\tb),r(\alpha))$ is the parabolic neighbourhood centered at $(\xb,\tb)$ of radius $r(\alpha)$. 
Then
\begin{equation}\label{as}
\frac{d}{dt}A_t<-\eta(\alpha)\hspace{1cm}\forall t\in[\tb-r(\alpha),\tb+r(\alpha)]
\end{equation}
for some $\eta>0$ only depending on $\alpha$.

By Lemma \ref{l1}, $A_t$ is positive and decreasing in time, and so property \eqref{as} can occur only on a finite number of time intervals, for any given $\alpha>0$. This shows that $|\phi(H)-h|$ tends to zero uniformly.

 In the area preserving case, setting $\bar{H}:=\frac{1}{A_t}\int_{\mt}H$ we can compute similarly 
\begin{align*}
\partial_t V_t&=\frac{1}{H}\int_{\mt}[\phi(H)-\phi(\bar{H})][H-\bar{H}]\\
&\geq\left(\inf\frac{\phi}{H}\right)\int_{\mt}[H-\bar{H}]^2
\end{align*}
then, as in the volume preserving case, we obtain that that $H$ converges uniformly to $\bar{H}$. This allows to pass  the limit under the integral sign, obtaining
\begin{align*}
\lim_{t\to\infty}\max_{\mt}|\phi(H)-h|&=\lim_{t\to\infty}\max_{\mt}|\phi(H)-\phi(\bar{H})+\phi(\bar{H})-h|\\
&\leq\lim_{t\to\infty}\max_{\mt}|\phi(H)-\phi(\bar{H})|+\lim_{t\to\infty}|\phi(\bar{H})-h|\\
&=\lim_{t\to\infty}\left|\phi(\bar{H})-\frac{\int_{\mt} H\phi(H)}{\int_{\mt} H}\right|=0
\end{align*}
using also the fact that $\phi$ is a continuous function. 
\end{proof}
Proposition \ref{hc} implies that any possible limit of subsequences of $\mt$ has constant mean curvature, and so, by a classical result of Alexandrov \cite{Al}, it is a geodesic sphere. Standard techniques, see e.g. \cite{An1,CaMi}, allow now to conclude that the whole family $\mt$ converges smoothly to a geodesic sphere  up to isometries.  
\subsection*{Exponential rate}
In order to prove the exponential rate, we will follow a similar method to \cite{GLW}, studying the behaviour of a suitable  ratio which is a modification of the ratio $Q=\frac{K}{H^n}$ used by Schulze in \cite{Sch2} for the Euclidean ambient space.  The analysis will be simpler than the analogous part in \cite{GLW} for two reasons. First, because $\phi$ is a function, even if non homogeneous, of the mean curvature, but especially because we already have the convergence to a geodesic sphere.
We will use the same notations as in \cite{GLW}. Let define the {\em perturbed} Weingarten operator  $\hti^i_{j}=h^i_j-a\delta^i_j$. Its trace, norm and determinant will be denoted as   $\Hti=\tr\hti^i_j$, $\Ati^2=\hti_i^j\hti_j^i$ and $\kti=\det\hti^i_{j}$ respectively. We also indicate by $\bti^i_j$ the inverse matrix of $\hti^i_j$, and by $\Bti$ its trace.
Then, we define
$$\qti=\frac{\kti}{\Hti}.$$ 
 Since the hypersurfaces $\mt$ approach uniformly to a geodesic sphere as $t$ goes to infinity, then $\mt$ is strictly $\hh$-convex for $t$ sufficiently large. Then $\qti$ is well defined and strictly positive for $t$ sufficiently large. This fact makes things work even if we choose an initial datum not necessarly  {\em strictly} $\hh$-convex, but just $\hh$-convex. Notice also that $\qti\leq\frac{1}{n^n}$, and the equality holds if and only if the hypersurface is totally umbilic. Then we know that $\qti$ converges smoothly to the constant value $\frac{1}{n^n}$. Our goal is to show that this convergence is exponential.
\begin{prop}\label{evoluzioni_tilde}
The quantities $\kti$, $\qti$ evolve according to
\begin{eqnarray*}
\partial_t\kti&=&\hspace{2mm}\phi'\Delta\kti-\frac{(n-1)}{n}\frac{\phi'|\nabla\kti|^2}{\kti}+
\frac{\kti}{\Hti^2}\phi'|\Hti\nabla\hti^i_j-\hti^i_j\nabla\Hti|_{g,\bti}^2\\
&&-\frac{\Hti^{2n}}{n\kti}\phi'\left|\nabla\frac{\kti}{\Hti^n}\right|^2
+\kti\phi''\bti^i_j\nabla_i\Hti\nabla^j\Hti+\kti\Hti(\phi-\phi'H-h)\\
&&+2an\kti(\phi-h)+n\kti\phi'\Ati^2+a\kti\phi'\Ati^2\Bti;\\
\partial_t\qti &=&\phi'\Delta\qti+\frac{(n+1)}{n\Hti^n}\phi'\langle\nabla\qti,\nabla\Hti^n\rangle
-\frac{(n-1)}{n\kti}\phi'\langle\nabla\qti,\nabla\kti\rangle\\
&&-\frac{\qti^{-1}}{n}\phi'|\nabla\qti|^2+\frac{\qti}{\Hti^2}\phi'|\Hti\nabla\hti^i_j-\hti^i_j\nabla\Hti|_{g,\bti}^2
+\qti\phi''|\nabla\Hti|^2_{\bti-\frac{n}{\Hti}g}\\
&&+\frac{\qti}{\Hti}(\phi'H-\phi+h)(n\Ati^2-\Hti^2)+a\qti\Ati^2\left(\Bti^2-\frac{n^2}{\Hti}\right).
\end{eqnarray*}
where
\begin{align*}
|\Hti\nabla\hti^i_j-\hti^i_j\nabla\Hti|_{g,\bti}^2&:=
b^n_mb^s_r(\Hti \nabla_i \hti^m_n -\hti^m_n\nabla_i\Hti)(\Hti\nabla^i\hti^r_s-\hti^r_s\nabla^i\Hti)\\
|\nabla\Hti|^2_{\bti-\frac{n}{\Hti}g}&:=\left(\bti^{ij}-\dfrac{n}{\Hti}g^{ij}\right)\nabla_i\Hti\nabla_j\Hti
\end{align*}
\end{prop}
\begin{proof}
Some trivial computations show that $\Hti=H-na$ and $\Ati^2=|A|^2+na^2-2aH$. Hence, by Proposition \ref{eq-ev}, we get
\begin{eqnarray*}
\partial_t\hti_i^j & =& \phi'\Delta\hti_i^j+\phi''\nabla_i H\nabla^j H+\left(\phi-h-\phi'H\right)\left(\hti_i^l\hti_l^j+2a\hti_i^j+a^2\delta_i^j\right)\\
&&+\phi'\left(\Ati^2+2a\Hti+2na^2\right)\left(\hti_i^j+a\delta_i^j\right)-a^2\left(\phi'H+\phi-h\right)\delta_i^j,\\
\partial_t\Hti^n & =& \phi'\Delta\Hti^n+n\Hti^{n-2}(\phi''\Hti-(n-1)\phi')|\nabla \Hti|^2+n\Hti^{n-1}\left(\phi-h\right)\left(\Ati^2+2a\Hti\right). 
\end{eqnarray*}

Moreover, using twice the law of the derivative of the determinant, we have 
\begin{eqnarray*}
\Delta\kti & = & \nabla_r\left(\kti\bti_j^i\nabla^r\hti_i^j\right)\\
 & = & \kti\bti_j^i\Delta\hti_i^j+\dfrac{|\nabla\kti|^2}{\kti}+\kti\nabla_r\bti_j^i\nabla^r\hti_i^j.
\end{eqnarray*}

We can use these equations to compute the evolution of $\kti$:
\begin{eqnarray*}
\partial_t\kti& = & \kti\bti_j^i\frac{\partial\hti_i^j}{\partial t}\\
&=&\phi'\Delta\kti-\phi'\dfrac{|\nabla\kti|^2}{\kti}-\kti\phi'\nabla_r\bti_j^i\nabla^r\hti_i^j+\kti\phi''\bti_j^i\nabla_i H\nabla^j H\\
&&+\kti\left(\phi-h-\phi'H\right)\left(\Hti+2an+a^2\Bti\right)-a^2\kti\Bti\left(\phi'H+\phi-h\right)\\
&&+\kti\phi'\left(\Ati^2+2a\Hti+2na^2\right)\left(n+a\Bti\right)\\
 & = & \phi'\Delta\kti-\phi'\dfrac{|\nabla\kti|^2}{\kti}-\kti\phi'\nabla_r\bti_j^i\nabla^r\hti_i^j+\kti\phi''\bti_j^i\nabla_i H\nabla^j H\\
&& +\kti\Hti(\phi-\phi'H-h)+2an\kti(\phi-h)+n\kti\phi'\Ati^2+a\kti\phi'\Ati^2\Bti.
\end{eqnarray*}
As in the proof of Lemma $2.1$ of \cite{Sch2}, we have
$$\kti\phi'\nabla_r\bti_j^i\nabla^r\hti_i^j=-\frac{1}{n}\frac{\phi'|\nabla\kti|^2}{\kti}-
\frac{\kti}{\Hti^2}\phi'|\Hti\nabla\hti^i_j-\hti^i_j\nabla\Hti|_{g,\bti}^2
+\frac{\Hti^{2n}}{n\kti}\phi'\left|\nabla\frac{\kti}{\Hti^n}\right|^2
$$
and then
\begin{eqnarray*}
-\phi'\dfrac{|\nabla\kti|^2}{\kti}-\kti\phi'\nabla_r\bti_j^i\nabla^r\hti_i^j &=&
-\frac{(n-1)}{n}\frac{\phi'|\nabla\kti|^2}{\kti}\\
&&+\frac{\kti}{\Hti^2}\phi'|\Hti\nabla\hti^i_j-\hti^i_j\nabla\Hti|_{g,\bti}^2
-\frac{\Hti^{2n}}{n\kti}\phi'\left|\nabla\frac{\kti}{\Hti^n}\right|^2.
\end{eqnarray*}
From the last equality, we get the evolution of $\kti$.
 
By definition of $\qti$, we get:
$$
\Delta\qti=\dfrac{\Delta\kti}{\Hti^n}-\dfrac{\kti\Delta\Hti^n}{\Hti^{2n}} -\dfrac{2}{\Hti^{2n}}\nabla_r\kti\nabla^r\Hti^n+2\dfrac{\qti}{\Hti^{2n}}|\nabla\Hti^n|^2.
$$
Then we have  
\begin{eqnarray*}
\partial_t\qti & =& \frac{1}{\Hti^n}\frac{\partial \kti}{\partial t}-\frac{n}{\Hti^{n+1}}\frac{\partial\Hti}{\partial t}\\
& = & \phi'\Delta\qti +\dfrac{2}{\Hti^{2n}}\nabla_r\kti\nabla^r\Hti^n-2\dfrac{\qti}{\Hti^{2n}}|\nabla\Hti^n|^2\\
&&+n(n-1)\frac{\qti}{\Hti^2}\phi'|\nabla\Hti|^2-\frac{n-1}{n}\phi'\frac{|\nabla\kti|^2}{\kti\Hti^n}\\
&&-\frac{\qti^{-1}}{n}\phi'|\nabla\qti|^2+\frac{\qti}{\Hti^2}\phi'|\Hti\nabla\hti^i_j-\hti^i_j\nabla\Hti|_{g,\bti}^2
+\qti\phi''|\nabla\Hti|^2_{\bti-\frac{n}{\Hti}g}\\
&&+\frac{\qti}{\Hti}(\phi'H-\phi+h)(n\Ati^2-\Hti^2)+a\qti\Ati^2\left(\Bti^2-\frac{n^2}{\Hti}\right).
\end{eqnarray*}
The conclusion follows observing that
\begin{eqnarray*}
\langle\nabla\qti,\nabla\Hti^n\rangle&=&\dfrac{1}{\Hti^n}\langle\nabla\kti,\nabla\Hti^n\rangle
-\dfrac{\qti}{\Hti^n}|\nabla\Hti^n|^2,\\
\langle\nabla\qti,\nabla\kti\rangle&=&
\dfrac{1}{\Hti^n}|\nabla\kti|^2-\dfrac{\qti}{\Hti^n}\langle\nabla\Hti^n,\nabla\kti\rangle.
\end{eqnarray*}
\end{proof}


Similarly as in \cite{GLW} we consider the function $f=\dfrac{1}{n^n}-\dfrac{\kti}{\Hti^n}$. 
 By the results of the previous section, we already know that $\mathcal{M}_t$ converges to a sphere and so $f$ converges smoothly to zero. Now we want to prove that this convergence is exponentially fast.

\begin{prop}
There are a time $\bar{t}>0$ and two constants $c,\ \delta>0$ such that for every time $t\geq\bar{t}$ we have : 
$$
f\leq ce^{-\delta t}.
$$
\end{prop}
\begin{proof}
 By Proposition \ref{evoluzioni_tilde}  we can compute the evolution equation for $f$: 
\begin{equation}\label{evf}
\begin{split}
\frac{\partial f}{\partial t} &= \phi'\Delta f-\frac{(n+1)}{n\Hti^n}\phi'\langle\nabla f,\nabla\Hti^n\rangle
+\frac{(n-1)}{n\kti}\phi'\langle\nabla f,\nabla\kti\rangle+\frac{\qti^{-1}}{n}\phi'|\nabla f|^2\\
&\hspace{4mm}-\frac{\qti}{\Hti^2}\phi'|\Hti\nabla\hti^i_j-\hti^i_j\nabla\Hti|_{g,\bti}^2
-\qti\phi''|\nabla\Hti|^2_{\bti-\frac{n}{\Hti}g}\\
 &\hspace{4mm}-a\phi'\qti\Ati^2\left(\Bti-\frac{n^2}{\Hti}\right)-\frac{\qti}{\Hti}\left(\phi' H-\phi+h\right)\left(n\Ati^2-\Hti^2\right).
 \end{split}
\end{equation}
First, we want to prove that the terms on the second line of the right member of \eqref{evf} give a negative contribution.
By Lemma $2.3$ part $ii)$ of \cite{Hu1}, we can show that, if there exists $\varepsilon\in(0,\frac{1}{n}]$  such that $\tilde{\lambda}_1\geq\varepsilon\Hti$, then 
\begin{equation}\label{grHu}
|\Hti\nabla\hti^i_j-\hti^i_j\nabla\Hti|^2\geq\frac{1}{2}\varepsilon^2\Hti^2|\nabla\Hti|^2
\end{equation}
For $t$ that grows to infinity, we can take $\varepsilon$ closer and closer to $\frac{1}{n}$. Proceeding as in the proof of Theorem $4.4$ of \cite{GLW}, we have
$$|\Hti\nabla\hti^i_j-\hti^i_j\nabla\Hti|_{g,\bti}^2\geq\frac{1}{\Hti^2}|\Hti\nabla\hti^i_j-\hti^i_j\nabla\Hti|^2$$
then, using \eqref{grHu},
\begin{equation}\label{gr1}
\frac{\qti}{\Hti^2}\phi'|\Hti\nabla\hti^i_j-\hti^i_j\nabla\Hti|_{g,\bti}^2\geq
\frac{1}{2}\varepsilon^2\frac{\qti}{\Hti^2}\phi'|\nabla\Hti|^2.
\end{equation}
Now we use estimate $(4.5)$ in \cite{GLW}: if $\varepsilon$ is close enough to $\frac{1}{n}$, then
$$\left|\bti^i_j-\frac{n}{\Hti}\delta^i_j\right|\leq\frac{\sqrt{n}(n-1)(1-n\varepsilon)}{(1-(n-1)\varepsilon)\Hti}.$$
Note that $\lim_{\varepsilon\to\frac{1}{n}}\frac{\sqrt{n}(n-1)(1-n\varepsilon)}{(1-(n-1)\varepsilon)\Hti}=0$.
Then, using also \eqref{gr1}, we have
\begin{align*}
\frac{\qti}{\Hti^2}&\phi'|\Hti\nabla\hti^i_j-\hti^i_j\nabla\Hti|_{g,\bti}^2
+\qti\phi''|\nabla\Hti|^2_{\bti-\frac{n}{\Hti}g}\\
&\geq\qti\left(\frac{\phi'}{\Hti^2}|\Hti\nabla\hti^i_j-\hti^i_j\nabla\Hti|_{g,\bti}^2
-\left|\bti^i_j-\frac{n}{\Hti}\delta^i_j\right|\phi''|\nabla\Hti|^2\right)\\
&\geq\qti\left(\frac{1}{2}\varepsilon^2\frac{\phi'}{\Hti^2}-
\phi''\frac{\sqrt{n}(n-1)(1-n\varepsilon)}{(1-(n-1)\varepsilon)\Hti}\right)|\nabla\Hti|^2.
\end{align*}
Then there exists a time $t_1$ such that for all $t\geq t_1$ holds
$$\frac{1}{2}\varepsilon^2\frac{\phi'}{\Hti^2}-\phi''\frac{\sqrt{n}(n-1)(1-n\varepsilon)}{(1-(n-1)\varepsilon)\Hti}\geq0.$$
Now we show that also the reaction terms are negative. By the relationship between the harmonic  and the arithmetic means of $n$ positive numbers, we have that $\Bti-\frac{n^2}{\Hti}\geq 0$. Therefore, the first term in the last line of \eqref{evf} can be omitted. Moreover, since we already know that $\mathcal{M}_t$ converges to a geodesic sphere, we can deduce some properties useful to estimate the remaining terms. By Proposition \ref{hc}, for every positive $\eta$, there exists a time $t_2>0$ such that for every $t\geq t_2$ the following holds:
$$
|\phi-h|\leq\eta.
$$
Then, if we choose $\eta$ small enough, we can find a positive constant $\delta_2$ such that
$$
\phi'H-\phi+h\geq \delta_2,
$$
when $t\geq t_2$.
Moreover there exists a time $t_3$ such that if $t\geq t_3$ $\mathcal{M}_t$ is strictly $\hh$-convex. It follows that there are two positive constants $\delta_3$ and $\delta_4$ such that for $t\geq t_3$ 
$$
\begin{array}{l}
\qti\Hti\geq\delta_3,\qquad\dfrac{n\Ati^2-\Hti^2}{\Hti^2}\geq\delta_4.
\end{array}
$$
The second inequality can be proved as in Lemma 2.5 of \cite{Sch2}.
Let $\bar{t}=\max\left\{t_1,t_2,t_3\right\}$ and $ \delta=\delta_2\delta_3\delta_4$, then for $t\geq\bar{t}$ we have
\begin{eqnarray*}
\frac{\partial f}{\partial t} &\leq& \phi'\Delta f-\frac{(n+1)}{n\Hti^n}\phi'\langle\nabla f,\nabla\Hti^n\rangle
+\frac{(n-1)}{n\kti}\phi'\langle\nabla f,\nabla\kti\rangle+\frac{\qti^{-1}}{n}\phi'|\nabla f|^2-\delta f.
\end{eqnarray*}
The thesis follow by the maximum principle.
\end{proof}

Arguing as in Theorem $3.5$ of \cite{Sch2} we obtain:
\begin{cor}
The second fundamental form $A$ converges exponentially in $C^{\infty}$ to the one of a geodesic sphere. In particular, there exist positive constants $c',\delta'$ such that
$$|\phi(H)-h|\leq c'e^{-\delta' t}$$
\end{cor}
From the previous corollary it follows that the limit hypersurface exists with no need to add isometries. In fact, for any $0\leq t_1<t_2$
\begin{align*}
\max_{\m}|F(x,t_1)-F(x,t_2)|&\leq \max_{\m}\int_{t_1}^{t_2}|\partial_t F(x,t)|dt\\
&=\max_{\m}\int_{t_1}^{t_2}|\phi-h|dt\leq\dfrac{c'}{\delta'}(e^{-\delta't_1}-e^{-\delta't_2})
\end{align*} 
then the whole family $F(\cdot,t)$ tends to a  limit hypersurface for $t$ that goes to infinity.
Finally, the smooth convergence of the second fundamental form implies the smooth convergence of the metric and of the embeddings, by standard arguments used for example in \cite{Sch2}.
 This complete the proof of Theorem \ref{mt}.


\bigskip
\noindent Maria Chiara Bertini, Dipartimento di Matematica e Fisica, Universit\`a di Roma ``Roma Tre'', Largo San Leonardo Murialdo 1, 00146, Roma, Italy. \\ E-mail: bertini@mat.uniroma3.it \\

\noindent Giuseppe Pipoli, Institut Fourier, Universit\'e Grenoble Alpes, 100 rue des maths, 38610, Gi\`eres, France.\\ E-mail: giuseppe.pipoli@univ-grenoble-alpes.fr\\
Giuseppe Pipoli was supported for this research by the ERC Avanced Grant 320939, Geometry and Topology of Open Manifolds (GETOM)

\end{document}